\documentclass[reqno,12pt]{amsart}

\usepackage{color}
\usepackage{enumerate}

\numberwithin{equation}{section}

\theoremstyle{plain}
\newtheorem{thm}{Theorem}[section]
\newtheorem{cor}[thm]{Corollary} 
\newtheorem{lemma}[thm]{Lemma} 

\theoremstyle{remark}

\newtheorem{remark}[thm]{Remark}
\newtheorem{remarks}[thm]{Remarks}

\theoremstyle{definition}

\def\question#1{\ifmmode\text{\bf \color{red} Q: #1}\else{\bf
    Q}\footnote{#1}\fi:\color{red}}

\usepackage{amscd,amssymb,comment,epic,eepic,euscript,graphics}

\newcommand{\la}{\lambda}

\newcommand\Reals{{\mathbb R}}
\newcommand\Complex{{\mathbb C}}
\newcommand\Nats{{\mathbb N}}

\newcommand{\BH}{\mathcal{B}(\mathcal{H})}
\newcommand{\Hcal}{\mathcal{H}}
\newcommand{\fS}{\mathcal{S}}
\newcommand\Tr{{\mathrm{Tr}}}
\newcommand{\st}{\,:\,}

\renewcommand{\i}{\text{\rm i}}
\newcommand{\supp}{\text{\rm supp}}

\begin{document}

\title{Asymptotic expansions for trace functionals}

\author[Skripka]{Anna Skripka$^{*}$} \email{skripka@math.unm.edu}

\thanks{\indent\llap{${}^{*}$}Research supported in part by NSF
  grant DMS-1249186.}

\address{A.S., Department of Mathematics and Statistics, University of New Mexico, 400 Yale Blvd NE, MSC01 1115, Albuquerque, NM 87131, USA}




\begin{abstract}
We obtain Taylor approximations for functionals $V\mapsto\Tr\big(f(H_0+V)\big)$ defined on the bounded self-adjoint operators, where $H_0$ is a self-adjoint operator with compact resolvent and $f$ is a sufficiently nice scalar function, relaxing assumptions on the operators made in \cite{vanS}, and derive estimates and representations for the remainders of these approximations.
\end{abstract}

\maketitle

\section{Introduction}
Let $H_0$ be an unbounded self-adjoint operator, $V$ a bounded self-adjoint operator on a separable Hilbert space $\mathcal{H}$, $f$ a sufficiently nice scalar function, and let $f(H_0+V)$ be defined by the standard functional calculus. The functionals $f\mapsto \Tr\big(f(H_0+V)\big)$ and $V\mapsto \Tr\big(f(H_0+V)\big)$ or their modifications
have been involved in problems of perturbation theory (of, for instance, differential operators) and noncommutative geometry since as early as 1950's (see, e.g., \cite{Azamov0,ACS,B,CC,DK,Hansen,L,vanS}). The latter functional in the context of noncommutative geometry is called the spectral (action) functional \cite{CC}.

Assume that the resolvent of $H_0$ belongs to some Schatten ideal (or, more generally, $\Tr\big(e^{-t H_0^2}\big)<\infty$, for any $t>0$), $\|\delta(V)\|<\infty$, $\|\delta^2(V)\|<\infty$, where $\delta(\cdot)=[|H_0|,\cdot]$, and $f$ is a sufficiently nice even function. Let $\{\mu_k\}_{k=1}^\infty$ be a sequence of eigenvalues of $H_0$ counting multiplicity and let $\{\psi_k\}_{k=1}^\infty$ be an orthonormal basis of the respective eigenvectors. The asymptotic expansion of the spectral action functional
\begin{align}
\label{2}
&\Tr\big(f(H_0+V)\big)=\Tr\big(f(H_0)\big)\\
\nonumber
&\quad+\sum_{p=1}^\infty \frac1p\sum_{i_1,\ldots,i_p}(f')^{[p-1]}(\mu_{i_1},\ldots,\mu_{i_p})\,
\left<V\psi_{i_1},\psi_{i_2}\right>\cdots\left<V\psi_{i_{p-1}},\psi_{i_p}\right>\left<V\psi_{i_p},\psi_{i_1}\right>,
\end{align}
where $(f')^{[p-1]}$ is the divided difference of order $p-1$ of the function $f'$, was derived in \cite{vanS}, extending the results of \cite{Hansen} for finite-dimensional operators. (The precise assumptions on $H_0$, $V$, and $f$ can be found in \cite[Theorem 18]{vanS}.)

In this paper, we obtain the asymptotic expansion \eqref{2} under relaxed assumptions on $H_0$ and $V$ and find bounds for the remainders of the respective approximations by taking a different approach to the problem. Specifically, we assume that $H_0=H_0^*$ has compact resolvent, $V=V^*\in\BH$  (where $\BH$ is the algebra of bounded linear operators on $\mathcal{H}$), and $f$ is a sufficiently nice compactly supported function (but no summability restriction on $H_0$ is made, $H_0$ is not assumed to be positive, and $f$ is not assumed to be even). Let
\begin{align}
\label{eq:sum}
&R_{H_0,f,n}(V):=\Tr\big(f(H_0+V)\big)-\Tr\big(f(H_0)\big)\\
\nonumber
&\quad-\sum_{p=1}^{n-1} \frac1p\sum_{i_1,\ldots,i_p}(f')^{[p-1]}(\mu_{i_1},\ldots,\mu_{i_p})\,
\left<V\psi_{i_1},\psi_{i_2}\right>\cdots\left<V\psi_{i_{p-1}},\psi_{i_p}\right>\left<V\psi_{i_p},\psi_{i_1}\right>.
\end{align}
In Theorem \ref{asexp} and Corollary \ref{asexpev}, we establish the bound
\begin{align}
\label{r1}
|R_{H_0,f,n}(V)|=\mathcal{O}\big(\|V\|^n\big)
\end{align}
and find an explicit estimate for $\mathcal{O}\big(\|V\|^n\big)$ in Theorem \ref{derbound} and Remark \ref{remarks}\eqref{remarksi}. (The case $n=1$ also follows from \cite{ACS}.)
If, in addition, $H_0$ has Hilbert-Schmidt resolvent, we refine the bound \eqref{r1} of Theorem \ref{asexp} in Theorem \ref{asexpL2}. In Theorem \ref{ac}, we show that the functional $C_c^{3}(\Reals)\ni f''\mapsto R_{H_0,f,2}(V)$ is given by a locally finite absolutely continuous measure. (An analogous result for the functional $f'\mapsto R_{H_0,f,1}(V)$) was obtained in \cite{ACS}.)

\section{Preliminaries}

The asymptotic expansion \eqref{2} can be rewritten as
\begin{align}
\label{3}
\nonumber
&\Tr\big(f(H_0+V)\big)=\Tr\big(f(H_0)\big)\\
&\quad+\sum_{p=1}^\infty \frac1p\sum_{\la_1,\ldots,\la_p\in \text{\rm spec}(H) }(f')^{[p-1]}(\la_1,\ldots,\la_p)\,
\Tr\big(E_{H_0}(\la_1)V\ldots VE_{H_0}(\la_p)V\big),
\end{align}
where $E_{H_0}$ is the spectral measure of $H_0=H_0^*$.

In this section, we justify that the traces in \eqref{3} are well defined and prepare a technical base for the derivation of \eqref{3}.
\bigskip


\paragraph{\bf Functional calculus}
We start with recalling some useful features of functional calculus for self-adjoint operators with compact resolvents.
Note that if the resolvent of an operator is compact at one point, then it is compact at all points of its domain. Note also that 
$(\i+H_0)^{-1}$ is compact if and only if $|(\i+H_0)^{-1}|=(1+H_0^2)^{-1/2}$ is compact.

By standard properties of the resolvent, we have

\begin{lemma} $($\cite[Lemma 1.3]{ACS}$)$
\label{compres}
If $H_0=H_0^*$ is defined in $\mathcal{H}$ and has compact resolvent and if $W=W^*\in\BH$, then $H_0+W$ also has compact resolvent.
\end{lemma}


\begin{lemma}$($\cite[Appendix B, Lemma 6]{CPh}$)$
\label{resest}
If $H_0=H_0^*$ is defined in $\mathcal{H}$ and if $W=W^*\in\BH$, then
\begin{align*}
\big(1+(H_0+W)^2\big)^{-1}\leq\left(1+\|W\|+\|W\|^2\right)(1+H_0^2)^{-1}.
\end{align*}
\end{lemma}

The following consequence was essentially established in \cite[Lemma 1.4]{ACS}.

\begin{cor}
\label{lemma:Efinite}
Let $H_0=H_0^*$ have compact resolvent and let $W=W^*$ be bounded. Then, for any compact subset $\delta$ of $\Reals$,
the spectral projection $E_{H_0+W}(\delta)$ has finite rank and
\begin{align}
\label{eq:Efinite}
E_{H_0+W}(\delta)\leq\left(1+\max_{s\in\delta}|s|^2\right)\cdot
\left(1+\|W\|+\|W\|^2\right)(1+H_0^2)^{-1}.
\end{align}
\end{cor}

\begin{proof}
From the spectral theorem we have
\[E_{H_0+W}(\delta)\leq\left(1+\max_{s\in\delta}|s|^2\right)\cdot\big(1+(H_0+W)^2\big)^{-1}.\]
Application of Lemma \ref{resest} gives \eqref{eq:Efinite}, which, in particular, implies that $E_{H_0+W}$ has finite rank.
\end{proof}

Note that for a compact subset $\delta$ of $\Reals$, $\Tr\big(E_H(\delta)\big)$ equals the number of eigenvalues of $H$, counting multiplicities, in the set $\delta$.

\begin{cor}
If $H_0=H_0^*$ satisfies $\big(1+H_0^2\big)^{-1/2}\in \fS^p$, with $p\geq 1$, and $W=W^*$ is bounded, then $\big(1+(H_0+W)^2\big)^{-1/2}\in \fS^p$.
\end{cor}

\begin{proof}
The result follows from Lemma \ref{resest} and operator monotonicity of the function $t\mapsto t^{p/2}$.
\end{proof}

Let $\fS^\alpha$ denote the Schatten ideal of order $\alpha$, that is,
\[\fS^\alpha=\left\{A\in\BH:\; \|A\|_\alpha:=\big(\Tr(|A|^\alpha)\big)^\frac{1}{\alpha}<\infty\right\}.\]

By standard properties of the Schatten norms, Corollary \ref{lemma:Efinite}, Theorem \ref{HSest}, and by the spectral theorem, we have the following lemma.
\begin{lemma}
\label{fH}
Let $H=H^*$ and let $f$ be a continuous compactly supported function on $\Reals$.
\begin{enumerate}[(i)]
\item
If $H$ has compact resolvent, then $f(H)\in \fS^1$ and
\begin{align}
\label{straight}
\|f(H)\|_1\leq \|f\|_\infty\cdot\Tr\big(E_H(\supp f)\big).
\end{align}
\item
Let $u(t)=(1+t^2)^{1/2}$.
If $(1+H^2)^{-1/2}\in \fS^2$, then 
\begin{align}
\label{straight2}
\|(fu)(H)\|_2\leq \|fu^2\|_\infty\cdot\big\|(1+H^2)^{-1/2}\big\|_2.
\end{align}
\end{enumerate}
\end{lemma}
\bigskip

\paragraph{\bf Operator derivatives}
Let $H_0=H_0^*$ be defined in $\mathcal{H}$ and let $V=V^*\in\BH$. Denote
\begin{align}
\label{eq:sumOp}
\mathcal{R}_{H_0,f,p}(V):=f(H_0+V)-\sum_{k=0}^{p-1}\frac{1}{k!}\cdot\frac{d^k}{ds^k}\bigg|_{s=0} f(H_0+sV),
\end{align}
provided the G\^{a}teaux derivatives exist in the operator norm. We will see in the proof of Theorem \ref{asexp} that $R_{H_0,f,p}(V)$ from \eqref{eq:sum} equals $\Tr\big(\mathcal{R}_{H_0,f,p}(V)\big)$.

Now we list results that guarantee the estimate
\begin{align*}
\|\mathcal{R}_{H_0,f,p}(V)\|=\mathcal{O}\big(\|V\|^p\big)
\end{align*}
for the operator norm of the remainder and help to establish the estimate \eqref{r1} for the trace of the remainder.

Recall that the divided difference of order $p$ is an operation on functions $f$ of one
(real) variable, which we will usually call $\la$, defined recursively as
follows:
\begin{align*}
&f^{[0]}(\la_0):=f(\la_0),\\
&f^{[p]}(\la_0,\ldots,\la_p):=
\begin{cases}
\frac{f^{[p-1]}(\la_0,\ldots,\la_{p-2},\la_{p-1})-
f^{[p-1]}(\la_0,\ldots,\la_{p-2},\la_p)}{\la_{p-1}-\la_p}&
\text{ if } \la_{p-1}\neq\la_p\\[2ex]
\frac{\partial}{\partial t}\big|_{t=\la_{p-1}}f^{[p-1]}(\la_0,\ldots,\la_{p-2},t)&
\text{ if } \la_{p-1}=\la_p.
\end{cases}
\end{align*}
Denote \[\mathcal{W}_p=\big\{f:\, f^{(j)},\widehat{f^{(j)}}\in L^1(\Reals), j=0,\ldots,p\big\}.\]
It is known (see, e.g., \cite[Lemma 2.3]{Azamov0}) that for $f\in\mathcal{W}_p$,
\begin{align*}
f^{[p]}(\la_0,\ldots,\la_p)=\int_{\Pi^{(p)}}e^{\i
(s_0-s_1)\la_0}e^{\i (s_1-s_2)\la_1}\dots e^{\i s_p \la_p}\,d\sigma_f^{(p)}(s_0,\dots,s_p),
\end{align*} where
\begin{align*}
&{\Pi^{(p)}}=\{(s_0,s_1,\dots,s_p)\in\Reals^{p+1}\st |s_p|\leq\dots\leq|s_1|\leq|s_0|,
\text{ \rm sign}(s_0)=\dots=\text{\rm sign}(s_p)\},\\
&d\sigma_f^{(p)}(s_0,s_1,\dots,s_p)=\i^p\hat f(s_0)\,ds_0\,ds_1\dots ds_p.
\end{align*}
Let $H_0=H_0^*,\ldots,H_p=H_p^*$  be defined in $\mathcal{H}$ and let $V_1,\ldots,V_p\in\BH$. If $f\in \mathcal{W}_p$, then the Bochner integral
\begin{align}
\label{trans}
T_{f^{[p]}}^{H_0,\ldots,H_p}(V_1,\ldots,V_p)y:=\int_{\Pi^{(p)}} e^{\i(s_0-s_1)H_0}V_1 e^{\i(s_1-s_2)H_1}V_2\ldots V_p\, e^{\i s_p H_p}y\,d\sigma_f^{(p)}(s_0,\dots,s_p)
\end{align}
exists for every $y\in\mathcal{H}$ and thus defined operator has the norm bound
\begin{align}
\label{nest}
\big\|T_{f^{[p]}}^{H_0,\ldots,H_p}(V_1,\ldots,V_p)\big\|\leq \frac{1}{p!}\cdot\big\|\widehat{f^{(p)}}\big\|_1\cdot\|V_1\|\cdot\ldots\cdot\|V_p\|
\end{align}
(see \cite[Lemma 4.5]{Azamov0}),
which follows from the bound for the total variation of the measure $\big\|\sigma_f^{(p)}\big\|\leq \frac{1}{p!}\big\|\widehat{f^{(p)}}\big\|_1$.

Similarly to \cite[Theorem 5.7]{Azamov0}, we have the following differentiation formula for an operator function.
\begin{thm}
\label{dermoi}
Let $H$ be a self-adjoint (unbounded) operator in $\mathcal{H}$, $V=V^*\in\BH$, $p\in\Nats$, and $f\in \mathcal{W}_p$.
Then,
\begin{align*}
\frac{1}{p!}\cdot\frac{d^p}{dt^p}\bigg|_{t=0} f(H+tV)=T_{f^{[p]}}^{H,\ldots,H}\underbrace{(V,\ldots,V)}_{p \text{ \rm times}}.
\end{align*}
\end{thm}

We will work with the subspace $G_p$ of $C^{p+1}(\Reals)$, $p\in\Nats$, consisting of functions $f$ such that $f^{(p)}, f^{(p+1)}\in L^2(\Reals)$. Let $\|\cdot\|_{G_p}$ denote the semi-norm
\[\|f\|_{G_p}=\frac{\sqrt{2}}{p!}\big(\big\|f^{(p)}\big\|_2+\big\|f^{(p+1)}\big\|_2\big).\]
It is known that $\frac{1}{p!}\big\|\widehat{f^{(p)}}\big\|_1\leq\|f\|_{G_p}$ (see, e.g., \cite[Lemma~7]{PS-Crelle}). In particular, we have $C_c^{p+1}(\Reals)\subset \cap_{j=0}^p G_j\subset \mathcal{W}_p$. Since all the scalar functions we consider are defined on $\Reals$, we will use the shortcut $C_c^{p+1}:=C_c^{p+1}(\Reals)$.

We will need the following version of the well known integral representation for the remainder of the Taylor approximation.
\begin{thm}
\label{integral}
If $f\in \cap_{j=0}^p G_j$, $H_0=H_0^*$ is defined in $\mathcal{H}$, and $V=V^*\in\BH$, then
\begin{align}
\label{eq:integral}
\mathcal{R}_{H_0,f,p}(V)=\frac{1}{(p-1)!}\int_0^1 (1-t)^{p-1}\frac{d^p}{ds^p}\bigg|_{s=t}f(H_0+sV)\,dt,
\end{align}
where the integral is defined for every $y\in\BH$ by
\begin{align*}
\left(\int_0^1 (1-t)^{p-1}\frac{d^p}{ds^p}\bigg|_{s=t}f(H_0+sV)\,dt\right) y
=\int_0^1 (1-t)^{p-1}\frac{d^p}{ds^p}\bigg|_{s=t}f(H_0+sV)\,y\,dt.
\end{align*}
\end{thm}

\begin{proof}
By Theorem \ref{dermoi}, the definition \eqref{trans}, the continuity of the function $t\mapsto e^{\i s(H+tV)}$, $s\in\Reals$, in the strong operator topology, Theorem \ref{moib}, and the Lebesgue dominated convergence for Bochner integrals, the function $t\mapsto \frac{d^n}{ds^n}\big|_{s=t}f(H_0+sV)$ is continuous in the strong operator topology. (More details in a slightly modified setting can be found in the proof of Theorem \ref{asexp} below.) For any $y,g\in\mathcal{H}$, define $\phi_{y,g}\in(\BH)^*$ by $\phi_{y,g}(A)=\left<Ay,g\right>$, for all $A\in\BH$. The function
\[t\mapsto \phi_{y,g}\left(\frac{d^n}{ds^n}\bigg|_{s=t}f(H_0+sV)\right)=\frac{d^n}{ds^n}\bigg|_{s=t}\phi_{y,g}\big(f(H_0+sV)\big)\]
is continuous. By the fundamental theorem of calculus and integration by parts,
\begin{align*}
\left<\mathcal{R}_{H_0,f,n}y,g\right>&=\frac{1}{(p-1)!}\int_0^1 (1-t)^{p-1}\frac{d^p}{ds^p}\bigg|_{s=t}\phi_{y,g}\big(f(H_0+sV)\big)\,dt\\
&=\left<\frac{1}{(p-1)!}\int_0^1 (1-t)^{p-1}\frac{d^p}{ds^p}\bigg|_{s=t}f(H_0+sV)y\,dt,g\right>,\quad
\text{for all } y,g\in\mathcal{H},
\end{align*}
completing the proof.
\end{proof}

The inequality \eqref{nest} has analogs for Schatten norms, as it is stated in \eqref{Spest} of the theorem below.

\begin{thm}
\label{moib}
Let $H=H^*$ be defined in $\mathcal{H}$, $p\in\Nats$, and $f\in \cap_{k=1}^p G_k$.
Let $\alpha,\alpha_1,\ldots,\alpha_p\in [1,\infty]$ be such that
$\frac{1}{\alpha}=\frac{1}{\alpha_1}+\ldots
+\frac{1}{\alpha_p}$ and let $V_j\in \fS^{\alpha_j}$, $j=1,\ldots,p$. Then,\footnote{$\|\cdot\|_\infty$ denotes the operator norm.}
\begin{align}
\label{Spest}
\big\|T_{f^{[p]}}^{H,\ldots,H}(V_1,\ldots,V_p)\big\|_\alpha\leq \|f\|_{G_p}\cdot\|V_1\|_{\alpha_1}\cdot\ldots\cdot\|V_p\|_{\alpha_p}.
\end{align}
\end{thm}

In the particular case of $p=1$ and $V$ a Hilbert-Schmidt perturbation, we have a stronger estimate, which holds for a more general $T_{f^{[1]}}^{H,H}(V)$ than the one defined above.
Let $\phi$ be a bounded continuous function on $\Reals^2$ and let
\begin{align}
\label{TR}
&\hat T_\phi^{H,H}(V)\\
\nonumber
&:=\|\cdot\|_2\,\text{-}\!\lim_{m\rightarrow\infty}\;\lim_{N\rightarrow\infty}\sum_{|l_0|,|l_1|\leq N}
\phi\left(\frac{l_0}{m},\frac{l_1}{m}\right)E_H\left(\left[\frac{l_0}{m},\frac{l_0+1}{m}\right)\right)
VE_H\left(\left[\frac{l_1}{m},\frac{l_1+1}{m}\right)\right).
\end{align}
The iterated limit above exists and defines a bounded operator on $\fS^2$ (with the bound as in the theorem below). The proof can be found on pp. 5--6 of \cite{HS-compatible} or, for a slightly different construction and more general $\phi$, in \cite{BS}.
\begin{thm}
\label{HSest}
Let $H$ be a self-adjoint (unbounded) operator in $\mathcal{H}$ and let $V\in \fS^2$. Then, for $\phi\in C_b(\Reals^2)$,
\[\big\|\hat T_{\phi}^{H,H}(V)\big\|_2\leq \|\phi\|_\infty\cdot\|V\|_2.\]
In particular, for $f\in C^1(\Reals)$, with $f'\in L^\infty(\Reals)$,
\[\big\|\hat T_{f^{[1]}}^{H,H}(V)\big\|_2\leq \|f'\|_\infty\cdot\|V\|_2.\]
\end{thm}

If $f\in\mathcal{W}_1$, then $\hat T_{f^{[1]}}^{H,H}(V)=T_{f^{[1]}}^{H,H}(V)$ by \cite[Lemma 3.5]{PSS}.

It is easy to see that if $V$ is a trace-class operator on $\Hcal$ and $E$ is a spectral measure (of a self-adjoint operator) acting on $\Hcal$, then the measure $\Tr\big(E(\cdot)V\big)$ has finite total variation. It is also known (see, e.g., \cite[Section 4]{ds} for references and details) that for $V_1,\ldots,V_p \in \fS^2$ and $E_1,\ldots,E_p$ spectral measures, with $p\geq 2$, the set function
\[A_1\times\cdots\times A_p\mapsto \Tr\big(E_1(A_1)V_1\ldots E_p(A_p)V_p\big),\]
where $A_1,\ldots,A_p$ are Borel subsets of $\Reals$, uniquely extends to a measure on $\Reals^p$ of finite total variation. These observations are core for the following useful representations for operator derivatives.

\begin{thm}
\label{moimeasure}
Let $H$ be a self-adjoint operator, $p\in\Nats$, and let $f\in \cap_{k=1}^p G_k$.
\begin{enumerate}[(i)]
\item
If $p=1$ and $V\in \fS^1$, then
\begin{align*}
\Tr\left(\frac{d}{dt}\bigg|_{t=0}f(H+tV)\right)=\int_\Reals f'(\la)\,\Tr\big(E_H(d\la)V\big).
\end{align*}
\item $($\cite[Theorem 3.12]{moissf}$)$
If $p\geq 2$ and $V\in \fS^2$, then
\begin{align*}
&\Tr\left(\frac{d^p}{dt^p}\bigg|_{t=0}f(H+tV)\right)\\
&\quad=(p-1)!\int_{\Reals^p}(f')^{[p-1]}(\la_1,\ldots,\la_p)\,\Tr\big(E_H(d\la_1)V\ldots E_H(d\la_p)V \big).
\end{align*}
\end{enumerate}
\end{thm}

It was proved in \cite[Lemma 3.5]{PSS} that for $f\in\mathcal{W}_p$ and $V_j\in \fS^{\alpha_j}$, $j=1,\ldots,p$, with $\alpha_1,\ldots,\alpha_p\in [1,\infty]$, the operator $T_{f^{[n]}}^{H,\ldots,H}(V_1,\ldots,V_p)$ given by \eqref{trans} coincides with the operator
\begin{align}
\label{TRmoi}
&\hat T_\phi^{H,\ldots,H}(V_1,\ldots,V_p)\\
\nonumber
&:=s\,\text{-}\!\lim_{m\rightarrow\infty}\;\|\cdot\|_\alpha\,\text{-}\!\lim_{N\rightarrow\infty}\sum_{|l_0|,\ldots,|l_p|\leq N}
\phi\left(\frac{l_0}{m},\frac{l_1}{m},\ldots,\frac{l_p}{m}\right)E_{H,l_0,m}V_1E_{H,l_1,m}V_2\ldots V_p E_{H,l_p,m},
\end{align}
where $\phi=f^{[p]}$, $E_{H,l_k,m}=E_H\left(\left[\frac{l_k}{m},\frac{l_k+1}{m}\right)\right)$,
for $k=0,\ldots,p$, $\frac{1}{\alpha}=\frac{1}{\alpha_1}+\ldots+\frac{1}{\alpha_p}$,
and $s$-$\lim$ denotes a limit in the strong operator topology on the tuples $(V_1,\ldots,V_p)\in \fS^{\alpha_1} \times \ldots \times \fS^{\alpha_p}$.
When $p=0$, we will use the symbol $\hat T_{f^{[0]}}$ (or $T_{f^{[0]}}$) to refer to the operator $\phi(H)$.

We need the following algebraic properties of $\hat T_\phi$, which can be derived straightforwardly from the definition \eqref{TRmoi}.

\begin{thm}
\label{properties}
Let $H=H^*$ be defined in $\mathcal{H}$. Let $\alpha,\alpha_1,\ldots,\alpha_p\in [1,\infty]$ be such that
$\frac{1}{\alpha}=\frac{1}{\alpha_1}+\ldots
+\frac{1}{\alpha_p}$ and let $V_j\in \fS^{\alpha_j}$, $j=1,\ldots,p$.
\begin{enumerate}[(i)]
\item \label{property1}
Let $\phi_1$ and $\phi_2$ be bounded Borel functions on $\Reals^p$. If the polylinear operators $\hat T_{\phi_1}^{H,\ldots,H},\hat T_{\phi_2}^{H,\ldots,H}:\fS^{\alpha_1}
    \times \ldots \times \fS^{\alpha_p}\mapsto\fS^\alpha$ exist and are bounded, then
    $\hat T_{\phi_1+\phi_2}^{H,\ldots,H}:\fS^{\alpha_1}
    \times \ldots \times \fS^{\alpha_p}\mapsto \fS^\alpha$ is also bounded and
    \[\hat T_{\phi_1+\phi_2}^{H,\ldots,H}=\hat T_{\phi_1}^{H,\ldots,H}+\hat T_{\phi_2}^{H,\ldots,H}.\]

\item $($\cite[Lemma 3.2(iii)]{PSS}$)$ \label{property2}
Let~$\phi_1: \Reals^{k + 1} \mapsto \Complex$ and~$\phi_2 : \Reals^{p - k
      + 1} \mapsto \Complex$ be bounded Borel functions such that the
    operators~$\hat T_{\phi_1}^{H,\ldots,H}$ and~$\hat T_{\phi_2}^{H,\ldots,H}$ exist and are bounded
    on~$\fS^{\alpha_1} \times \ldots \times \fS^{\alpha_k}$ and~$\fS^{\alpha_{k +
        1}} \times \ldots \times \fS^{\alpha_p}$, respectively.  If $$
    \phi(\lambda_0, \ldots, \lambda_p) := \phi_1 \left( \lambda_0,
      \ldots, \lambda_k \right) \cdot \phi_2 \left( \lambda_k, \ldots,
      \lambda_p \right), $$ then
      the operator~$\hat T_{\phi}^{H,\ldots,H}:\fS^{\alpha_1}
    \times \ldots \times \fS^{\alpha_p}\mapsto \fS^\alpha$ is bounded and $$\hat T_{\phi}^{H,\ldots,H} \left( V_1,
      \ldots, V_p \right) = \hat T_{\phi_1}^{H,\ldots,H} \left( V_1, \ldots, V_k \right)
    \cdot \hat T_{\phi_2}^{H,\ldots,H} \left( V_{k + 1}, \ldots, V_p
    \right). $$

\item $($\cite[Lemma 2.9]{HS-compatible}$)$ \label{property3}
Let $\phi:\Reals^{p}\mapsto\Complex$ and $\psi_1,\psi_2:\Reals\mapsto\Complex$ be bounded Borel functions.
Denote
\[(\psi_1\phi\psi_2)(\la_0,\ldots,\la_p):=\psi_1(\la_0)\phi(\la_0,\ldots,\la_p)\psi_2(\la_p).\]
If $\hat T_\phi^{H,\ldots,H}:\fS^{\alpha_1} \times \ldots \times \fS^{\alpha_p}\mapsto\fS^\alpha$ exists and is bounded, then the operator
$\hat T_{\psi_1\phi\psi_2}^{H,\ldots,H}:\fS^{\alpha_1} \times \ldots \times \fS^{\alpha_p}\mapsto\fS^\alpha$ is also bounded and
\[\hat T_{\psi_1\phi\psi_2}^{H,\ldots,H}(V_1,\ldots,V_p)=\hat T_\phi^{H,\ldots,H}(\psi_1(H)V_1,\ldots,V_p\psi_2(H)).\]
\end{enumerate}
\end{thm}

\section{Asymptotic expansions}
In this section we prove the Taylor asymptotic expansion \eqref{2} for $H_0$ having compact resolvent and find bounds for the remainder $R_{H_0,f,n}$.
\\

\paragraph{\bf Compact resolvent}
We start with deriving estimates for the transformations \eqref{trans}, which will imply estimates for directional operator derivatives and the remainders $\mathcal{R}_{H_0,f,n}(V)$ defined in \eqref{eq:sumOp}.

\begin{lemma}\label{constants}
Let $H=H^*$ be defined in $\mathcal{H}$ and have compact resolvent and let $V=V^*\in\BH$. Denote\footnote{As usually,
$x\mapsto \lfloor x\rfloor$ denotes the floor function.} 
$j_n=1+\lfloor\log_2(n)\rfloor$.
Then, for each function $0\leq f \in C_c^{n+1}$ with $f^{2^{-j_n}}\in C_c^{n+1}$, the transformation $T_{f^{[n]}}^{H,\ldots,H}$ is a bounded polylinear operator from $\BH\times\cdots\times \BH$ to $\fS^1$ and
\begin{align}
\label{compactest}
\left\|T_{f^{[n]}}^{H,\ldots,H}(V,\ldots,V)\right\|_1&\leq a_n\cdot \|V\|^n\cdot  \Tr\big(E_H(\supp f)\big)\\
\nonumber
&\quad\times\max_{1\leq k\leq j_n}\big\|f^{2^{-k}}\big\|_\infty\cdot\left(\max_{\substack{1\leq k\leq j_n\\1\leq d\leq n}}\left\{1,\;\big\|f^{2^{-k}}\big\|_{G_d}\right\}\right)^n,
\end{align}
where\footnote{$\{a_n\}_{n=1}^\infty=\{2,4,6,10,14,20,26,36,46,60,74,94,114,140,\ldots\}$}
\begin{align}
\label{an}
a_1=2,\quad a_k=\begin{cases}a_{k-1}+a_{\frac{k}{2}} & \text{if } k \text{ is even}\\
a_{k-1}+a_{\frac{k-1}{2}} & \text{if } k\geq 3 \text{ is odd}.\end{cases}
\end{align}
\end{lemma}

\begin{proof}
Note that by the Leibnitz formula for the divided difference,
\begin{align}
\label{obs0}
f^{[n]}=\big(\sqrt{f}\cdot\sqrt{f}\big)^{[n]}
&=\sum_{k=0}^{\lfloor\frac{n-1}{2}\rfloor}
\sqrt{f}^{[k]}(\la_0,\ldots,\la_k)\sqrt{f}^{[n-k]}(\la_k,\ldots,\la_n)\\
\nonumber
&\quad+\sum_{k=0}^{\lfloor\frac{n-1}{2}\rfloor}\sqrt{f}^{[n-k]}(\la_0,\ldots,\la_{n-k})\sqrt{f}^{[k]}(\la_{n-k},\ldots,\la_n)\\
\nonumber
&\quad+\begin{cases}
\sqrt{f}^{[\frac{n}{2}]}(\la_0,\ldots,\la_{\frac{n}{2}})\sqrt{f}^{[\frac{n}{2}]}(\la_{\frac{n}{2}},\ldots,\la_n) &
\text{if } n \text{ is even}\\
0 & \text{if } n \text{ is odd}.
\end{cases}
\end{align}
Hence, by Theorem \ref{properties} (and the equality $\hat T_{f^{[n]}}=T_{f^{[n]}}$), we have
\begin{align}
\label{obsT}
& T_{f^{[n]}}^{H,\dots,H}(V,\ldots,V)=\sum_{k=0}^{\lfloor\frac{n-1}{2}\rfloor}
T_{\sqrt{f}^{[k]}}^{H,\ldots,H}(V,\ldots,V)\cdot T_{\sqrt{f}^{[n-k]}}^{H,\ldots,H}(V,\ldots,V)\\
\nonumber
&\quad+\sum_{k=0}^{\lfloor\frac{n-1}{2}\rfloor}T_{\sqrt{f}^{[n-k]}}^{H,\ldots,H}(V,\ldots,V)\cdot
T_{\sqrt{f}^{[k]}}^{H,\ldots,H}(V,\ldots,V)+\begin{cases}
\left(T_{\sqrt{f}^{[\frac{n}{2}]}}^{H,\ldots,H}(V,\ldots,V)\right)^2 &
\text{if } n \text{ is even}\\
0 & \text{if } n \text{ is odd}.
\end{cases}
\end{align}
Recall that when $k=0$, the operator $T_{\sqrt{f}^{[k]}}^{H,\ldots,H}(V,\ldots,V)$ degenerates to the operator  $\sqrt{f}(H)$.

If $n=1$, then \eqref{obsT} reduces to
\begin{align}
\label{T1}
T_{f^{[1]}}^{H,H}(V)=\sqrt{f}(H)\cdot T_{\sqrt{f}^{[1]}}^{H,H}\big(V\big)+T_{\sqrt{f}^{[1]}}^{H,H}\big(V\big)\cdot\sqrt{f}(H).
\end{align}
From Theorem \ref{moib} and the straightforward inequality \eqref{straight} applied to $\sqrt{f}(H)$, we derive
\begin{align}
\label{T1n}
\left\|T_{f^{[1]}}^{H,H}(V)\right\|_1\leq 2\cdot\big\|\sqrt{f}\big\|_\infty\cdot\big\|\sqrt{f}\big\|_{G_1}
\cdot\|V\|\cdot\Tr\big(E_H(\supp f)\big).
\end{align}

If $n=2$, then \eqref{obsT} reduces to
\begin{align}
\label{T2}
T_{f^{[2]}}^{H,H,H}(V,V)&=\sqrt{f}(H)\cdot T_{\sqrt{f}^{[2]}}^{H,H,H}(V,V)+T_{\sqrt{f}^{[2]}}^{H,H,H}(V,V)\cdot \sqrt{f}(H)\\
\nonumber
&\quad+T_{\sqrt{f}^{[1]}}^{H,H}(V)\cdot T_{\sqrt{f}^{[1]}}^{H,H}(V).
\end{align}
Hence,
\begin{align*}
\left\|T_{f^{[2]}}^{H,H,H}(V,V)\right\|_1\leq 2\,\big\|\sqrt{f}(H)\big\|_1\cdot \left\|T_{\sqrt{f}^{[2]}}^{H,H,H}(V,V)\right\|+\left\|T_{\sqrt{f}^{[1]}}^{H,H}(V)\right\|_1\cdot \left\|T_{\sqrt{f}^{[1]}}^{H,H}(V)\right\|.
\end{align*}
Applying, in addition, Theorem \ref{moib} and the estimates \eqref{straight} and \eqref{T1n}, we obtain
\begin{align}
\label{T2n}
\left\|T_{f^{[2]}}^{H,H,H}(V,V)\right\|_1 &\leq
4\,\|V\|^2\cdot\Tr\big(E_H(\supp f)\big)\cdot\max\left\{\big\|\sqrt{f}\big\|_\infty, \big\|\sqrt[4]{f}\big\|_\infty\right\}\\
\nonumber
&\quad\times\left(\max\left\{1,\;\big\|\sqrt{f}\big\|_{G_1}, \big\|\sqrt[4]{f}\big\|_{G_1}, \big\|\sqrt{f}\big\|_{G_2}\right\}\right)^2.
\end{align}

Application of Theorem \ref{properties} and the decomposition \eqref{obsT} gives
\begin{align}
\label{rec}
\left\|T_{f^{[n]}}^{H,\ldots,H}(V,\ldots,V)\right\|_1
&\leq 2\sum_{k=0}^{\lfloor\frac{n-1}{2}\rfloor}\left\|T_{\sqrt{f}^{[k]}}^{H,\ldots,H}(V,\ldots,V)\right\|_1
\cdot\left\|T_{\sqrt{f}^{[n-k]}}^{H,\ldots,H}(V,\ldots,V)\right\|\\
\nonumber
&\quad+\begin{cases}
\left\|T_{\sqrt{f}^{[\frac{n}{2}]}}^{H,\ldots,H}(V,\ldots,V)\right\|_1 \cdot
\left\|T_{\sqrt{f}^{[\frac{n}{2}]}}^{H,\ldots,H}(V,\ldots,V)\right\| &
\text{if } n \text{ is even}\\
0 & \text{if } n \text{ is odd}.
\end{cases}
\end{align}
where the involved transformations are bounded by Theorem \ref{moib}.
We will prove by induction on $n$ that the right hand side RHS of \eqref{rec} satisfies
\begin{align}
\label{indineq}
\text{RHS} \leq a_n\cdot \|V\|^n\cdot  \Tr\big(E_H(\supp f)\big)\cdot \max_{1\leq k\leq j_n}\big\|f^{2^{-k}}\big\|_\infty\cdot\left(\max_{\substack{1\leq k\leq j_n\\1\leq d\leq n}}\left\{1,\;\big\|f^{2^{-k}}\big\|_{G_d}\right\}\right)^n.
\end{align}

Suppose that the estimate \eqref{indineq} is proved for $n-1$ (and for all $1\leq m\leq n-1$). Then we have

\begin{align}
\label{obs1}
&\left\|T_{\sqrt{f}^{[p]}}^{H,\ldots,H}(V,\ldots,V)\right\|_1\cdot
\left\|T_{\sqrt{f}^{[q]}}^{H,\ldots,H}(V,\ldots,V)\right\|\\
\nonumber
&\quad\leq a_p\cdot\Tr\big(E_H(\supp f)\big)\cdot\|V\|^n\cdot\max_{1\leq k\leq j_n}\big\|f^{2^{-k}}\big\|_\infty\cdot\left(\max_{\substack{1\leq k\leq j_n\\1\leq d\leq n}}\left\{1,\;\big\|f^{2^{-k}}\big\|_{G_d}\right\}\right)^n,
\end{align}
where $p=\frac{n-1}{2}$, $q=\frac{n+1}{2}$ if $n$ is odd and $p=q=\frac{n}{2}$ if $n$ is even.
Similarly, we have the bound
\begin{align}
\label{obs2}
&\left\|T_{f^{[n]}}^{H,\ldots,H}(V,\ldots,V)-T_{\sqrt{f}^{[p]}}^{H,\ldots,H}(V,\ldots,V)\cdot T_{\sqrt{f}^{[q]}}^{H,\ldots,H}(V,\ldots,V)\right\|_1\\
\nonumber
&\quad \leq 2\sum_{k=0}^{\lfloor\frac{n-2}{2}\rfloor}\left\|T_{\sqrt{f}^{[k]}}^{H,\ldots,H}(V,\ldots,V)\right\|_1
\cdot\left\|T_{\sqrt{f}^{[n-k]}}^{H,\ldots,H}(V,\ldots,V)\right\|\\
\nonumber
&\quad\leq a_{n-1}\cdot\Tr\big(E_H(\supp f)\big)\cdot\|V\|^n\cdot \max_{1\leq k\leq j_n}\big\|f^{2^{-k}}\big\|_\infty\cdot\left(\max_{\substack{1\leq k\leq j_n\\1\leq d\leq n}}\left\{1,\;\big\|f^{2^{-k}}\big\|_{G_d}\right\}\right)^n.
\end{align}
Combining \eqref{obs1} and \eqref{obs2} completes the proof of the estimate.

The value of $j_n$ is defined as follows. We repeat recursively the decomposition \eqref{obsT} until each summand in the sum representing $T_{f^{[n]}}^{H,\ldots,H}(V,\ldots,V)$ decomposes into a product of $f^{2^{-i}}(H)\in \fS^1$ (see Lemma \ref{fH}) and operators in the form $T_{(f^{2^{-l}})^{[m]}}^{H,\ldots,H}(V,\ldots,V)$, for some $1\leq i, l\leq j_n$ and $1\leq m\leq n$. We have derived in \eqref{T1n} and \eqref{T2n} that $j_1=1$ and $j_2=1+j_1$. By the analogous reasoning, $j_n=1+j_{\lfloor\frac{n}{2}\rfloor}=\ldots=r+j_{\lfloor\frac{n}{2^r}\rfloor}$. The recursive procedure stops when
$\lfloor\frac{n}{2^r}\rfloor=1$. Hence, $j_n=1+\lfloor\log_2(n)\rfloor$.
\end{proof}

\begin{thm}
\label{derbound}
Let $H_0=H_0^*$ be defined in $\mathcal{H}$ and have compact resolvent and let $V=V^*\in\BH$.
Then, for each function $f \in C_c^{n+1}$,
\begin{align}
\label{O}
\nonumber
&\left\|\frac{1}{n!}\cdot\frac{d^n}{ds^n}\bigg|_{s=t}f(H_0+sV)\right\|_1\leq C_{H_0,f,n}\cdot\Tr\big(E_{H_0+tV}(\supp f)\big)\cdot\|V\|^n,\quad t\in [0,1],\\
&\big|\Tr\left(\mathcal{R}_{H_0,f,n}(V)\right)\big|\leq C_{H_0,f,n}\cdot\sup_{t\in [0,1]}\Tr\big(E_{H_0+tV}(\supp f)\big)\cdot\|V\|^n.
\end{align}
\end{thm}

\begin{proof}
Decompose the function $f$ into $f=f_1-f_2$, where $0\leq f_i,f_i^{2^{-j_n}}\in C_c^{n+1}$ for $i=1,2$. From Theorem \ref{dermoi} we have
\begin{align*}
\frac{1}{n!}\cdot\frac{d^n}{ds^n}\bigg|_{s=t}f(H_0+sV)=
T_{f_1^{[n]}}^{H_0+tV,\ldots,H_0+tV}(V,\ldots,V)-T_{f_2^{[n]}}^{H_0+tV,\ldots,H_0+tV}(V,\ldots,V).
\end{align*}
By Lemma \ref{compres}, $H_0+tV$ has compact resolvent. Hence, for each $T_{f_i^{[n]}}$, $i=1,2$, we have the bound as in \eqref{compactest} of Lemma \ref{constants}.

Note that the bound for $\mathcal{R}_{H_0,f,n}(V)$ would follow from the integral representation for the remainder
\begin{align}
\label{eq:integralTr}
\Tr\big(\mathcal{R}_{H_0,f,n}(V)\big)=\frac{1}{(n-1)!}\int_0^1 (1-t)^{n-1}\,\Tr\left(\frac{d^n}{ds^n}\bigg|_{s=t}f(H_0+sV)\right)dt
\end{align}
and the estimate for the derivatives established above. By the argument given in the proof of Theorem \ref{integral},
the functions
\[t\mapsto \frac{d^n}{ds^n}\bigg|_{s=t}f(H_0+sV)\quad\text{and}\quad t\mapsto \left(\frac{d^n}{ds^n}\bigg|_{s=t}f(H_0+sV)\right)^*\] are continuous in the strong operator topology. These functions are also uniformly $\fS^1$-bounded; therefore,
\eqref{eq:integral} implies \eqref{eq:integralTr} on the strength of \cite[Lemma 3.10]{Azamov0}.
\end{proof}

\begin{remarks}
\label{remarks}
\begin{enumerate}[(i)]
\item \label{remarksi}
If $f\geq 0$ and $f^{2^{-j_n}}\in C_c^{n+1}$, then
\begin{align*}
C_{H_0,f,n}&\leq a_n\cdot\max_{1\leq k\leq j_n}\big\|f^{2^{-k}}\big\|_\infty\cdot\left(\max_{\substack{1\leq k\leq j_n\\1\leq d\leq n}}\left\{1,\;\big\|f^{2^{-k}}\big\|_{G_d}\right\}\right)^n,
\end{align*}
where $a_n$ is given by \eqref{an}.

\item The case $n=1$ was handled in \cite{ACS} and it inspired decomposition of $f$ into positive and negative parts and use of dyadic roots of $f$ in the proof of Lemma \ref{constants}. It was established in \cite[Theorem 1.23]{ACS} that the function $f: H\in H_0+(\BH)_{sa}\mapsto f(H)$, with $f\in C_c^3$, is Fr\'{e}chet differentiable and the derivative is continuous in the trace norm. The $n$th order Fr\'{e}chet differentiability in the trace norm can also be established, provided we take $f\in C_c^{n+2}$.
\end{enumerate}
\end{remarks}

\begin{thm}
\label{asexp}
Let $H_0=H_0^*$ be defined in $\mathcal{H}$ and have compact resolvent and $V=V^*\in\BH$.
Then, for each function $f \in C_c^{n+1}$,
\begin{align}
\label{eq:asexp}
&\Tr(f(H_0+V))\\
\nonumber
&\quad=\Tr(f(H_0))+\sum_{k=1}^{n-1}\;\frac1k\sum_{\la_1,\ldots,\la_k\in\text{\rm spec}(H_0)} (f')^{[k-1]}(\la_1,\ldots,\la_k)\,
\Tr\big(E_{H_0}(\la_1)V\ldots E_{H_0}(\la_k)V\big)\\
\nonumber
&\quad\quad+\Tr\big(\mathcal{R}_{H_0,f,n}(V)\big),
\end{align}
with
\[\Tr\big(\mathcal{R}_{H_0,f,n}(V)\big)=\mathcal{O}\big(\|V\|^n\big)\] satisfying \eqref{O}.
\end{thm}

\begin{proof}
By Theorem \ref{dermoi},
\[f(H_0+V)=f(H_0)+\sum_{k=1}^{n-1}T_{f^{[k]}}^{H_0,\ldots,H_0}(V,\ldots,V)
+\mathcal{R}_{H_0,f,n}(V),\]
where each summand is in $\fS^1$ by Theorem \ref{derbound}.
Hence,
\[\Tr(f(H_0+V))=\Tr(f(H_0))+\sum_{k=1}^{n-1}\Tr\big(T_{f^{[k]}}^{H_0,\ldots,H_0}(V,\ldots,V)\big)
+\Tr\big(\mathcal{R}_{H_0,f,n}(V)\big).\]
The bound for the remainder is provided by Theorem \ref{derbound}, so we are left to prove the representation
\begin{align}
\label{obs3}
\nonumber
&\Tr\big(T_{f^{[k]}}^{H_0,\ldots,H_0}(V,\ldots,V)\big)\\
&\quad=\frac1k\sum_{\la_1,\ldots,\la_k\in\text{\rm spec}(H_0)} (f')^{[k-1]}(\la_1,\ldots,\la_k)\,
\Tr\big(E_{H_0}(\la_1)V\ldots E_{H_0}(\la_k)V\big),
\end{align}
for any $k=1,\ldots,n-1$.

Let $E_m:=E_{H_0}([-m,m])$. Clearly, $E_m$ converges to the identity in the strong operator topology and, by Corollary \ref{lemma:Efinite}, $V_m:=E_m V E_m\in \fS^1$. Theorem \ref{moimeasure} implies
\begin{align}
\label{obs4}
\nonumber
&\Tr\big(T_{f^{[k]}}^{H_0,\ldots,H_0}(V_m,\ldots,V_m)\big)\\
&\quad=\frac1k\sum_{\substack{\la_1,\ldots,\la_k\in\text{\rm spec}(H_0)\\ |\la_1|,\ldots,|\la_k| \leq m}} (f')^{[k-1]}(\la_1,\ldots,\la_k)\,
\Tr\big(E_{H_0}(\la_1)V\ldots E_{H_0}(\la_k)V\big).
\end{align}

As it was noted in the proof of Lemma \ref{constants}, $T_{f^{[k]}}^{H_0,\ldots,H_0}(V,\ldots,V)$ is decomposable into a sum where each summand is a product of $f^{2^{-i}}(H_0)\in \fS^1$ and operators in the form $T_{(f^{2^{-l}})^{[p]}}^{H_0,\ldots,H_0}(V,\ldots,V)$, for some $1\leq i, l\leq j_n$ and $1\leq p\leq k$. We also have the completely analogous decomposition for $T_{f^{[k]}}^{H_0,\ldots,H_0}(V_m,\ldots,V_m)$.
Firstly, we verify that $T_{(f^{2^{-l}})^{[p]}}^{H_0,\ldots,H_0}(V_m,\ldots,V_m)$ converges to $T_{(f^{2^{-l}})^{[p]}}^{H_0,\ldots,H_0}(V,\ldots,V)$ in the strong operator topology as $m\rightarrow\infty$ by the Lebesgue dominated convergence theorem for Bochner integrals. Indeed, define
\begin{align*}
&h(\omega)=e^{\i(s_0-s_1)H}V e^{\i(s_1-s_2)H}V\ldots V e^{\i s_p H},\\
& h_m(\omega)=e^{\i(s_0-s_1)H}V_m e^{\i(s_1-s_2)H}V_m\ldots V_m e^{\i s_p H}.
\end{align*}
Then we have convergence of the integrands in \eqref{trans}
\[\lim_{m\rightarrow\infty}h_m(\omega) y=h(\omega)y,\quad
\text{for every }\omega=(s_0,\dots,s_p)\in\Omega,\] and we also have
\[\sup_m\|h_m(\cdot)\|\in L^1\big(\Omega,\sigma_f^{(p)}\big),\] which implies
\begin{align}
\label{lprop}
\lim_{m\rightarrow\infty}T_{(f^{2^{-l}})^{[p]}}^{H_0,\ldots,H_0}(V_m,\ldots,V_m)y
=T_{(f^{2^{-l}})^{[p]}}^{H_0,\ldots,H_0}(V,\ldots,V)y,\quad y\in\mathcal{H}.
\end{align}
Since we have uniform boundedness
\[\sup_m\big\|T_{(f^{2^{-l}})^{[p]}}^{H_0,\ldots,H_0}(V_m,\ldots,V_m)\big\|\leq \big\|f^{2^{-l}}\big\|_{G^p}\|V\|^p\] (see Theorem \ref{moib}), the convergence in \eqref{lprop} along with $f^{2^{-i}}(H_0)\in \fS^1$ implies that
\[\lim_{m\rightarrow\infty}\Tr\big(T_{f^{[k]}}^{H_0,\ldots,H_0}(V_m,\ldots,V_m)\big)= \Tr\big(T_{f^{[k]}}^{H_0,\ldots,H_0}(V,\ldots,V)\big)\] (see, e.g., \cite[Lemma 2.5]{Azamov0}), which also implies convergence of the sequence on the right hand side of \eqref{obs4} to the expression on the right hand side of \eqref{obs3}. Thus, \eqref{obs3} is proved.
\end{proof}

Since \eqref{2} can be written as \eqref{3}, we have the following consequence of Theorem \ref{asexp}.

\begin{cor}
\label{asexpev}
Let $H_0=H_0^*$ be defined in $\mathcal{H}$ and have compact resolvent and let $V=V^*\in\BH$. Let $\{\mu_k\}_{k=1}^\infty$ be a sequence of eigenvalues of $H_0$ counting multiplicity and let $\{\psi_k\}_{k=1}^\infty$ be an orthonormal basis of the respective eigenvectors. Then, for each function $f \in C_c^{n+1}$,
\begin{align*}
&\Tr\big(f(H_0+V)\big)-\Tr\big(f(H_0)\big)\\
\nonumber
&\quad=\sum_{p=1}^{n-1} \frac1p\sum_{i_1,\ldots,i_p}(f')^{[p-1]}(\mu_{i_1},\ldots,\mu_{i_p})\,
\left<V\psi_{i_1},\psi_{i_2}\right>\cdots\left<V\psi_{i_p},\psi_{i_1}\right>+\Tr\big(\mathcal{R}_{H_0,f,n}(V)\big),
\end{align*}
with
\[\Tr\big(\mathcal{R}_{H_0,f,n}(V)\big)=\mathcal{O}\big(\|V\|^n\big)\] satisfying \eqref{O}.
\end{cor}

\paragraph{\bf Hilbert-Schmidt resolvent}
Under the assumption $(1+H_0^2)^{-1/2}\in \fS^2$, in Theorem \ref{asexpL2},
we improve the bound for the remainder obtained in Corollary \ref{asexpev} by eliminating $\sup\limits_{t\in [0,1]}\Tr\big(E_{H_0+tV}(\supp f)\big)$ and, consequently, eliminating $\sup\limits_{t\in [0,1]}\max\limits_{s\in\supp f}(1+|s|^2)$ (see connection between these expressions in \eqref{eq:Efinite}).

\begin{lemma}\label{CL2}
Let $H=H^*$ satisfy $(1+H^2)^{-1/2}\in \fS^2$ and let $V=V^*$ be bounded. Denote $u(t)=(1+t^2)^{1/2}$. Then, for every $n\in\Nats$ and $f\in C_c^{n+1}$, the transformation $T_{f^{[n]}}^{H,\ldots,H}$ is a bounded polylinear mapping from $\BH\times\cdots\times \BH$ to $\fS^1$ and
\[\left\|T_{f^{[n]}}^{H,\ldots,H}(V,\ldots,V)\right\|_1\leq c_{f,n}\cdot\big\|(1+H^2)^{-\frac12}\big\|_2^2\cdot\|V\|^n,\]
where
\begin{align}
\label{eq:CL2}
c_{f,n}\leq
\begin{cases}
\|fu^2\|_{G_1}+2\,\|fu^2\|_\infty\quad\text{\rm if } n=1\\
\|fu^2\|_{G_n}+\frac{n(n+3)}{2}\max\limits_{1\leq k\leq n}\big\{\|f\|_\infty,\|fu\|_\infty,
\|f\|_{G_k},\|fu\|_{G_k}\big\}\cdot\max\limits_{2\leq l\leq n}\|u\|_{G_l}^2\quad\text{\rm if } n\geq 2.
\end{cases}
\end{align}
\end{lemma}

We need the following routine lemma.

\begin{lemma}
\label{longL}
Let $f,u\in C^n$. Then,
\begin{align*}
&u(\la_0)\,f^{[n]}(\la_0,\ldots,\la_n)\,u(\la_n)\\
&\quad=(fu^2)^{[n]}(\la_0,\ldots,\la_n)-\psi_1(\la_0,\ldots,\la_n)
-\psi_2(\la_0,\ldots,\la_n)-\psi_3(\la_0,\ldots,\la_n),
\end{align*}
where
\begin{align*}
\psi_1(\la_0,\ldots,\la_n)&=\sum_{k=1}^n (fu)^{[n-k]}(\la_0,\ldots,\la_{n-k})\,u^{[k]}(\la_{n-k},\ldots,\la_n),
\end{align*}
\begin{align*}
\psi_2(\la_0,\ldots,\la_n)&=\sum_{k=1}^n u^{[k]}(\la_0,\ldots,\la_k)\,(fu)^{[n-k]}(\la_k,\ldots,\la_n),
\end{align*}
\begin{align*}
\psi_3(\la_0,\ldots,\la_n)&=\sum_{k=1}^{n-1} u^{[k]}(\la_0,\ldots,\la_k)\sum_{j=1}^{n-k}\,f^{[n-k-j]}(\la_k,\ldots,\la_{n-j})\,u^{[j]}(\la_{n-j},\ldots,\la_n).
\end{align*}
\end{lemma}

\begin{proof}
By the Leibnitz formula for the divided difference,
\begin{align*}
&u(\la_0)\,f^{[n]}(\la_0,\ldots,\la_n)\,u(\la_n)\\
&=(uf)^{[n]}(\la_0,\ldots,\la_n)\,u(\la_n)-\sum_{k=1}^n u^{[k]}(\la_0,\ldots,\la_k)\,f^{[n-k]}(\la_k,\ldots,\la_n)\,u(\la_n),
\end{align*}
and applying the Leibnitz formula one more time completes the proof.
\end{proof}

\begin{proof}[Proof of Lemma \ref{CL2}]
It is easy to see that $fu^2,fu,f\in G_n$, for any natural $n$, and $u\in G_k$, for any $k\geq 2$.
Note also that $\|u'\|_\infty\leq 1$.

Denote $\widetilde V:=(1+H^2)^{-1/2}V(1+H^2)^{-1/2}\in \fS^1$. For brevity, we denote the function $(\la_0,\ldots,\la_n)\mapsto u(\la_0)\,f^{[n]}(\la_0,\ldots,\la_n)\,u(\la_n)$ by $uf^{[n]}u$.
In case $n=1$, Lemma \ref{longL} and Theorem \ref{properties}, along with the equality $\hat T_{f^{[1]}}=T_{f^{[1]}}$, ensure the decomposition
\begin{align*}
T_{f^{[1]}}^{H,H}(V)&=\hat T_{u f^{[1]}u}^{H,H}(\widetilde V)\\
&=T_{(fu^2)^{[1]}}^{H,H}(\widetilde V)-\big((fu)(H)\big)\cdot \hat T_{u^{[1]}}^{H,H}(\widetilde V)
-\hat T_{u^{[1]}}^{H,H}(\widetilde V)\cdot\big((fu)(H)\big).
\end{align*}
Theorem \ref{HSest} implies
\[\left\|\hat T_{u^{[1]}}(\widetilde V)\right\|_2\leq \|u'\|_\infty\|\widetilde V\|_2\leq \|\widetilde V\|_2.\]
Applying also Theorem \ref{moib} and Lemma \ref{fH} gives
\begin{align*}
\left\|T_{f^{[1]}}^{H,H}(V)\right\|_1 &\leq \|fu^2\|_{G_1}\cdot\big\|(1+H^2)^{-1/2}V(1+H^2)^{-1/2}\big\|_1\\
&\quad+2\|fu^2\|_\infty\cdot \big\|(1+H^2)^{-1/2}\big\|_2\cdot\big\|(1+H^2)^{-1/2}V(1+H^2)^{-1/2}\big\|_2
\end{align*}

Let now $n\geq 2$ and denote $W=(1+H^2)^{-1/2}V$. Since the operator $H$ is fixed, to lighten the notation, we omit the superscript when refer to the transformation $T_{f^{[n]}}(V,\ldots,V)$ and similar ones.
Applying Lemma \ref{longL} and Theorem \ref{properties} leads to the decomposition
\begin{multline}
\label{Tfn}
T_{f^{[n]}}(V,\ldots,V)=T_{u f^{[n]}u}(W,V,\ldots,V,W^*)
=T_{(fu^2)^{[n]}}(W,V,\ldots,V,W^*)\\-T_{\psi_1}(W,V,\ldots,V,W^*)
-T_{\psi_2}(W,V,\ldots,V,W^*)-T_{\psi_3}(W,V,\ldots,V,W^*),
\end{multline}
where
\begin{align*}
&T_{\psi_1}(W,V,\ldots,V,W^*)=-T_{(fu)^{[n-1]}}(W,V,\ldots,V)\cdot \hat T_{u^{[1]}}(W^*)\\
&\quad-\sum_{k=2}^{n-1} T_{(fu)^{[n-k]}}(W,V,\ldots,V)\, T_{u^{[k]}}(V,\ldots,V,W^*)
-(fu)(H)\cdot T_{u^{[n]}}(W,\ldots,W^*),
\end{align*}
\begin{align*}
&T_{\psi_2}(W,V,\ldots,V,W^*)=-\hat T_{u^{[1]}}(W)\cdot T_{(fu)^{[n-1]}}(V,\ldots,V,W^*)\\
&\quad\quad-\sum_{k=2}^{n-1} T_{u^{[k]}}(W,V,\ldots,V)\cdot T_{(fu)^{[n-k]}}(V,\ldots,V,W^*)
-T_{u^{[n]}}(W,\ldots,W^*)\cdot(fu)(H),
\end{align*}
\begin{align*}
&T_{\psi_3}(W,V,\ldots,V,W^*)=-\hat T_{u^{[1]}}(W)\cdot T_{f^{[n-2]}}(V,\ldots,V)\cdot \hat T_{u^{[1]}}(W^*)\\
&\quad-\hat T_{u^{[1]}}(W)\cdot\sum_{j=2}^{n-1}T_{f^{[n-1-j]}}(V,\ldots,V)\cdot T_{u^{[j]}}(V,\ldots,V,W^*)\\
&\quad-\sum_{k=2}^{n-1} T_{u^{[k]}}(W,V,\ldots,V)\cdot T_{f^{[n-1-k]}}(V,\ldots,V)\cdot \hat T_{u^{[1]}}(W^*)\\
&\quad-\sum_{k=2}^{n-2} T_{u^{[k]}}(W,V,\ldots,V)\cdot\sum_{j=2}^{n-k}T_{f^{[n-k-j]}}(V,\ldots,V)\cdot
T_{u^{[j]}}(V,\ldots,V,W^*).
\end{align*}
Application of Theorems \ref{moib} and \ref{HSest} implies the bounds
\begin{align}
\label{e1}
\left\|T_{(fu^2)^{[n]}}(W,V,\ldots,V,W^*)\right\|_1\leq\|fu^2\|_{G_n}\cdot\|W\|_2^2 \cdot\|V\|^{n-2},
\end{align}
\begin{align}
\label{e2}
\nonumber
&\left\|T_{\psi_i}(W,V,\ldots,V,W^*)\right\|_1\\
\nonumber
&\quad\leq \bigg(\|fu\|_{G_{n-1}}+\sum_{k=2}^{n-1}\|fu\|_{G_{n-k}}\|u\|_{G_k}+\|fu\|_\infty\|u\|_{G_n}\bigg)\cdot\|W\|_2^2
\cdot\|V\|^{n-2}\\
&\quad\leq n\cdot\max_{1\leq k\leq n-1}\big\{\|fu\|_\infty,\|fu\|_{G_k}\big\}\cdot\max_{2\leq l\leq n}\|u\|_{G_l}\cdot \|W\|_2^2\cdot\|V\|^{n-2},
\end{align}
for $i=1,2$, and
\begin{multline}
\label{e3}
\left\|T_{\psi_3}(W,V,\ldots,V,W^*)\right\|_1
\leq \bigg(\|f\|_{G_{n-2}}+\sum_{j=2}^{n-1}\|f\|_{G_{n-1-j}}\|u\|_{G_j}
+\sum_{k=2}^{n-1}\|u\|_{G_k}\|f\|_{G_{n-1-k}}\\
+\sum_{k=2}^{n-2}\|u\|_{G_k}\sum_{j=2}^{n-k}\|f\|_{G_{n-k-j}}\|u\|_{G_j}\bigg)\cdot\|W\|_2^2\cdot\|V\|^{n-2}\\
\leq \bigg(n-1+n-2+\sum_{k=2}^{n-2}(n-k-1)\bigg)\cdot\max_{1\leq k\leq n-2}\big\{\|f\|_\infty,\|f\|_{G_k}\big\}\cdot\max_{2\leq l\leq n-1}\|u\|_{G_l}^2\cdot \|W\|_2^2\cdot\|V\|^{n-2},
\end{multline}
where $\|f\|_{G_0}$ stands for $\|f\|_\infty$.

Combining \eqref{e1} - \eqref{e3} gives \eqref{eq:CL2}.
\end{proof}


\begin{thm}
\label{asexpL2}
Let $H_0=H_0^*$ be defined in $\mathcal{H}$, let $V=V^*\in\BH$, and suppose that $(1+H_0^2)^{-1/2}\in \fS^2$. Let $\{\mu_k\}_{k=1}^\infty$ be a sequence of eigenvalues of $H_0$ counting multiplicity and let $\{\psi_k\}_{k=1}^\infty$ be an orthonormal basis of the respective eigenvectors. Then, for $n\in\Nats$ and $f\in C_c^{n+1}$,
\begin{align*}
&\Tr\big(f(H_0+V)\big)-\Tr\big(f(H_0)\big)\\
\nonumber
&\quad=\sum_{p=1}^{n-1} \frac1p\sum_{i_1,\ldots,i_p}(f')^{[p-1]}(\mu_{i_1},\ldots,\mu_{i_p})\,
\left<V\psi_{i_1},\psi_{i_2}\right>\cdots\left<V\psi_{i_p},\psi_{i_1}\right>+\Tr\big(\mathcal{R}_{H_0,f,n}(V)\big)
\end{align*} and
\begin{align*}
\left|\Tr\big(\mathcal{R}_{H_0,f,n}(V)\big)\right|\leq c_{f,n}\cdot\big\|(1+H_0^2)^{-1}\big\|_1\cdot
\left(1+\|V\|+\|V\|^2\right)\cdot\|V\|^n,
\end{align*}
where $c_{f,n}$ is as in \eqref{eq:CL2}.
\end{thm}

\begin{proof}
The result follows upon applying Lemma \ref{resest} to $W=tV$, Lemma \ref{CL2} to $H=H_0+tV$, $t\in[0,1]$, repeating the approximation argument in the proof of Theorem \ref{asexp}, and using the integral representation for the remainder as in the proof of Theorem \ref{derbound}.
\end{proof}

We conclude with the discussion of the integral representations for $R_{H_0,V,1}(f)$ and $R_{H_0,V,2}(f)$. Let $C_c^3((a,b))$ denote the set of $C^3$-functions whose closed supports are compact subsets of $(a,b)$.

\begin{thm}$($\cite[Theorem 2.5]{ACS}$)$\label{R1}
Let $H_0=H_0^*$ have a compact resolvent and let $V=V^*\in\BH$. Then, for $f\in C_c^3((a,b))$,
\begin{align*}
\Tr\big(f(H_0+V)\big)=\Tr\big(f(H_0)\big)+\int_\Reals f'(\la)\Tr\big(E_{H_0}((a,\la])-E_{H_0+V}((a,\la])\big)\,d\la.
\end{align*}
\end{thm}

\begin{proof}
Applying the spectral theorem, Corollary \ref{lemma:Efinite}, and performing integration by parts gives
\begin{align*}
&\Tr\big(f(H_0+V)\big)-\Tr\big(f(H_0)\big)\\
&\quad=\int_\Reals f(\la)\,d\,\Tr\big(E_{H_0+V}((a,\la])\big)
-\int_\Reals f(\la)\,d\,\Tr\big(E_{H_0}((a,\la])\big)\\
&\quad=\int_\Reals f'(\la)\Tr\big(E_{H_0}((a,\la])-E_{H_0+V}((a,\la])\big)\,d\la.
\end{align*}
\end{proof}

\begin{thm}
\label{ac}
Let $H_0=H_0^*$ satisfy $(1+H_0^2)^{-1/2}\in \fS^2$ and let $V=V^*\in\BH$. Denote $u(t)=(1+t^2)^{1/2}$. Then, there is a locally integrable function
$\eta=\eta_{H_0,V}$ such that
\begin{align}
\label{trf}
R_{H_0,f,2}(V)=\int_\Reals f''(t)\eta(t)\,dt,\quad\text{for }f\in C_c^3,
\end{align}
and
\[\int_{[a,b]}|\eta(t)|\,dt\leq C_{a,b}\cdot\|(1+H_0^2)^{-1}\big\|_1\cdot
\left(1+\|V\|+\|V\|^2\right)\cdot\|V\|^2,\]
where
\begin{align*}
C_{a,b}\leq 9\cdot\max\big\{1,(b-a)^2\big\}\cdot\max\big\{2,\|u\|_{L^\infty([a,b])},\|u^2\|_{L^\infty([a,b])},
\|(u^2)'\|_{L^\infty([a,b])}\big\}
\end{align*}
and $R_{H_0,f,2}(V)$ is given by \eqref{eq:sum}.
\end{thm}

\begin{proof}
Let $H_t=H_0+tV$ and $W_t=(1+H_t^2)^{-1/2}V$, for $t\in [0,1]$. As a particular case of \eqref{Tfn}, we have
\begin{align*}
T_{f^{[2]}}^{H_t,H_t,H_t}(V,V)&=T_{(fu^2)^{[2]}}^{H_t,H_t,H_t}(W_t,W_t^*)\\
&\quad-\hat T_{(fu)^{[1]}}^{H_t,H_t}(W_t)\cdot \hat T_{u^{[1]}}^{H_t,H_t}(W_t^*)-(fu)(H)\cdot T_{u^{[2]}}^{H_t,H_t,H_t}(W_t,W_t^*)\\
&\quad-\hat T_{u^{[1]}}^{H_t,H_t}(W_t)\cdot \hat T_{(fu)^{[1]}}^{H_t,H_t}(W_t^*)- T_{u^{[2]}}^{H_t,H_t,H_t}(W_t,W_t^*)\cdot (fu)(H)\\
&\quad-\hat T_{u^{[1]}}^{H_t,H_t}(W_t)\cdot f(H)\cdot \hat T_{u^{[1]}}^{H_t,H_t}(W_t^*).
\end{align*}
Therefore, with application of Theorem \ref{HSest}, H\"{o}lder's inequality, and $\|u'\|_\infty\leq 1$ we have
\begin{align}
\label{Tf2}
\nonumber
\left|\Tr\left(T_{f^{[2]}}^{H_t,H_t,H_t}(V,V)\right)\right|&\leq \left|\Tr\left(T_{(fu^2)^{[2]}}^{H_t,H_t,H_t}(W_t,W_t^*)\right)\right|\\
\nonumber
&+2\cdot\|(fu)'\|_\infty\cdot\|W_t\|_2^2+2\cdot\|fu\|_\infty\cdot
\left|\Tr\left(T_{u^{[2]}}^{H_t,H_t,H_t}(W_t,W_t^*)\right)\right|\\
&+\|f\|_\infty\cdot\|W_t\|_2^2.
\end{align}
One can derive from \eqref{trans} that
\[\Tr\left(T_{u^{[2]}}^{H_t,H_t,H_t}(W_t,W_t^*)\right)=\Tr\left(T_\phi^{H_t,H_t}(W_t)W_t^*\right),
\quad\text{with}\quad\phi(\la_0,\la_1)=u^{[2]}(\la_0,\la_0,\la_1)\]
(for more details, see, e.g., \cite[Lemma 3.8]{moissf}). Hence, by H\"{o}lder's inequality, the equality $T_\phi=\hat T_\phi$ (see \cite[Lemma 3.5]{PSS}), and Theorem \ref{HSest},
\begin{align}
\label{u2}
\nonumber
\left|\Tr\left(T_{u^{[2]}}^{H_t,H_t,H_t}(W_t,W_t^*)\right)\right|&\leq \big\|\hat T_\phi^{H_t,H_t}(W_t)\big\|_2\cdot\|W_t\|_2
\leq \big\|u^{[2]}\big\|_\infty\cdot\|W_t\|_2^2\\&\leq\frac12\cdot\|u''\|_\infty\cdot\|W_t\|_2^2\leq\|W_t\|_2^2.
\end{align}
Similarly,
\begin{align}
\label{fu2}
\left|\Tr\left(T_{(fu^2)^{[2]}}^{H_t,H_t,H_t}(W_t,W_t^*)\right)\right|\leq \frac12\cdot\|(fu^2)''\|_\infty\cdot\|W_t\|_2^2.
\end{align}
Since
\begin{align*}
\big\|(fg)^{(k)}\big\|_\infty=\left\|\sum_{j=0}^k\begin{pmatrix}k\\j\end{pmatrix}f^{(j)}g^{(k-j)}\right\|_\infty
&\leq \sum_{j=0}^k \begin{pmatrix}k\\j\end{pmatrix}\big\|f^{(j)}\big\|_\infty\big\|g^{(k-j)}\big\|_\infty\\
&\leq 2^k\cdot\max_{0\leq j\leq k}\big\|f^{(j)}\big\|_\infty\cdot\max_{0\leq l\leq k}\big\|g^{(l)}\big|\big\|_\infty,
\end{align*}
\[\big\|f^{(j)}\big\|_\infty\leq \big\|f^{(n)}\big\|_\infty\cdot (b-a)^{n-j},\quad 0\leq j\leq n,\]
and
\[(u^2)''\equiv 2,\]
we have that for $f\in C_c^3((a,b))$,
\begin{align}
\label{fi}
&\|(fu)'\|_\infty\leq 2\cdot\|f''\|_\infty\cdot\max\big\{1,(b-a)^2\big\}\cdot\|u\|_{L^\infty([a,b])},\\
\nonumber
&\|(fu^2)''\|_\infty\leq 4\cdot\|f''\|_\infty\cdot\max\big\{1,(b-a)^2\big\}
\cdot\max\big\{2,\|u^2\|_{L^\infty([a,b])},\|(u^2)'\|_{L^\infty([a,b])}\big\}.
\end{align}
Combination of the inequalities \eqref{Tf2} - \eqref{fi} ensures the bound
\begin{multline}
\label{of}
\left|\Tr\left(T_{f^{[2]}}^{H_t,H_t,H_t}(V,V)\right)\right|\leq \|f''\|_\infty\cdot\|W_t\|_2^2\\
\times
9\cdot\max\big\{1,(b-a)^2\big\}\cdot\max\big\{2,\|u\|_{L^\infty([a,b])},\|u^2\|_{L^\infty([a,b])},
\|(u^2)'\|_{L^\infty([a,b])}\big\}.
\end{multline}
Applying \eqref{of}, Lemma \ref{resest}, and Theorems \ref{integral} and \ref{dermoi} gives
\begin{align}
\label{RC}
\big|R_{H_0,f,2}(V)\big|\leq C_{a,b}\cdot\big\|f''\big\|_\infty\cdot\big\|(1+H_0^2)^{-1}\big\|_1\cdot
\left(1+\|V\|+\|V\|^2\right)\cdot\|V\|^2.
\end{align}
Hence, by the Riesz representation theorem for a functional in $\big(C_c(\Reals)\big)^*$, there is a locally finite measure $\nu=\nu_{H_0,V}$, with \[\int_{[a,b]}d|\nu|\leq C_{a,b}\cdot\|(1+H_0^2)^{-1}\big\|_1\cdot
\left(1+\|V\|+\|V\|^2\right)\cdot\|V\|^2,\] such that
\begin{align}
\label{Rtr}
R_{H_0,f,2}(V)=\int_\Reals f''(t)\,d\nu(t),\quad\text{for }f\in C_c^3.
\end{align}

We are left to prove absolute continuity of $\nu$.
By adjusting the proof of \cite[Theorem 2.25]{HS-compatible}, we derive the representation
\begin{align*}
T_{f^{[1]}}^{H_0,H_0}(V)=\hat T_F^{H_0,H_0}\big((1+H_0^2)^{-1/2}V\big)\big((1+H_0^2)^{-1/2},\quad\text{for } f\in C_c^3((a,b)),
\end{align*}
where
\[F(\la_0,\la_1)=u(\la_0)f^{[1]}(\la_0,\la_1)u(\la_1),\quad \|F\|_\infty\leq \widetilde C_{a,b}\cdot\|f'\|_\infty.\]
Hence, by Theorem \ref{dermoi}, H\"{o}lder's inequality, and Theorem \ref{HSest},
\begin{align*}
\left|\Tr\left(\frac{d}{dt}\bigg|_{t=0}f(H_0+tV)\right)\right|\leq \widetilde C_{a,b}\cdot\big\|f'\big\|_\infty\cdot\big\|(1+H_0^2)^{-1}\big\|_1\cdot\|V\|.
\end{align*}
Therefore, there exists a locally finite measure $\mu=\mu_{H_0,V}$ such that
\begin{align}
\label{Dtr}
\Tr\left(\frac{d}{dt}\bigg|_{t=0}f(H_0+tV)\right)=\int_\Reals f'(t)\,d\mu(t),\quad\text{for }f\in C_c^3.
\end{align}
Let
\begin{align}
\label{xic}
\xi(\la):=\Tr\big(E_{H_0}((a,\la])-E_{H_0+V}((a,\la])\big).
\end{align}
By Theorem \ref{R1} and by \eqref{Dtr}, we have
\begin{align*}
R_{H_0,f,2}(V)&=\Tr\big(f(H_0+V)\big)-\Tr\big(f(H_0)\big)-\Tr\left(\frac{d}{dt}\bigg|_{t=0}f(H_0+tV)\right)\\
&=\int_\Reals f'(\la)\xi(\la)\,d\la-\int_\Reals f'(\la)\,d\mu(\la).
\end{align*}
Integrating by parts yields
\begin{align*}
R_{H_0,f,2}(V)=\int_\Reals f''(\la)\left(\mu((a,\la))-\int_a^\la \xi(t)\,dt\right)d\la,\quad\text{for }f\in C_c^3((a,b)),
\end{align*}
completing the proof of the absolute continuity of $\nu$.
\end{proof}

\begin{remark}
Analogs of the function $\xi$ given be \eqref{xic} and the function $\eta$ given by \eqref{trf} have long history in perturbation theory. It was established in \cite{Krein}, \cite{Kop}, and \cite{PSS} for $n=1$, $n=2$, and $n\geq 3$, respectively,
that there exists an integrable function $\eta_n=\eta_{n,H_0,V}$ such that
\begin{align*}
\Tr\big(\mathcal{R}_{H_0,f,n}(V)\big)=\int_\Reals f^{(n)}(t)\eta_n(t)\,dt
\end{align*} for sufficiently nice functions $f$ (including $f\in C_c^{n+1}$), provided $V=V^*\in \fS^n$ and $H_0=H_0^*$ (without restrictions on the resolvent of $H_0$). If the resolvent of $H_0$ is compact and $V\in\BH$, then $\xi$ essentially coincides with the spectral flow (see \cite{ACS}). An analog of \eqref{trf} (with substantially modified left hand side)
was obtained for $H_0=H_0^*$ and $V=V^*$ satisfying $(1+H_0^2)^{-1/2}V\in\fS^2$ in \cite[Theorem 4.9]{HS-compatible}.
\end{remark}

\bibliographystyle{plain}

\begin{thebibliography}{99}
\bibitem{Azamov0}
N.~A.~Azamov, A.~L.~Carey, P.~G.~Dodds, F.~A.~Sukochev,
{\it Operator integrals, spectral shift, and spectral flow,} Canad.
J. Math. {\bf 61} (2009), no. 2, 241--263.

\bibitem{ACS}
N.~A.~Azamov, A.~L.~Carey, F.~A.~Sukochev,
{\it The spectral shift function and spectral flow,} Comm. Math.
Phys. {\bf 276} (2007), no.~1, 51--91.

\bibitem{BS}
M. Sh. Birman, M. Solomyak, {\it Double operator integrals in a Hilbert space,} Integral Equations Operator Theory {\bf 47} (2003), no. 2, 131-–168.

\bibitem{B}
J.-M. Bouclet, {\it Spectral distributions for long range perturbations,} J. Funct. Anal. {\bf 212} (2004), 431--471.

\bibitem{CPh}
A.~L.~Carey, J.~Phillips, {\it Unbounded Fredholm modules and spectral flow,} Canad. J. Math. {\bf 50} (1998), 673–-718.

\bibitem{CC}
A. H. Chamseddine, A. Connes, {\it The spectral action principle,} Comm. Math. Phys. {\bf 186} (1997), 731--750.

\bibitem{DK}
Yu. L. Daleckii, S. G. Krein, {\it Integration and differentiation of functions of Hermitian operators and applications to the theory of perturbations,} (Russian) Voronež. Gos. Univ. Trudy Sem. Funkcional. Anal. {\bf 1956} (1956), no. 1, 81-–105.

\bibitem{ds}
K. Dykema, A. Skripka, {\it Higher order spectral shift,} J. Funct. Anal.
{\bf 257} (2009), 1092--1132.

\bibitem{Hansen}
F. Hansen, {\it Trace functions as Laplace transforms,} J. Math. Phys. {\bf 47} (2006) 043504, 11 pp.

\bibitem{Krein} M.~G.~Krein, \emph{On a trace formula in perturbation
theory}, Matem. Sbornik {\bf 33} (1953), 597--626 (Russian).

\bibitem{Kop} L.~S.~Koplienko, \emph{Trace formula for perturbations
    of nonnuclear type}, Sibirsk. Mat. Zh.  {\bf 25} (1984), 62--71
  (Russian). English transl.\ in Siberian Math. J. {\bf 25} (1984),
  735--743.

\bibitem{L}
I. M. Lifshits, {\it On a problem of the theory of perturbations connected with quantum statistics,} Uspehi Matem. Nauk (N.S.), {\bf 7} (1952), no. 1 (47), 171--180. (Russian)

\bibitem{PSS}
D.~Potapov, A.~Skripka, F.~Sukochev, {\it Spectral shift function of higher order,} Invent. Math., {\bf 193} (2013), no. 3, 501--538.

\bibitem{HS-compatible}
D.~Potapov, A.~Skripka, F.~Sukochev, {\it On Hilbert-Schmidt compatibility,} Oper. Matrices, {\bf 7} (2013), no. 1, 1--34.

\bibitem{PS-Crelle} D.~Potapov and F.~Sukochev, \emph{Unbounded
Fredholm modules and double operator integrals}, J. reine. angew.
Math. {\bf 626} (2009), 159--185.


\bibitem{moissf} A.~Skripka, {\it Multiple operator integrals and
spectral shift,} Illinois J. Math., 55 (2011), no. 1, 305--324.

\bibitem{vanS}
W. D. van Suijlekom, {\it Perturbations and operator trace functions,} J. Funct. Anal. {\bf 260} (2011), no. 8, 2483–-2496.

\end{thebibliography}

\end{document}